\documentclass[12pt,reqno]{amsart}
\usepackage{amsmath}
\usepackage{amsfonts}
\usepackage{amssymb}
\usepackage{amsthm}
\usepackage{amscd}
\usepackage{amsgen}
\usepackage{latexsym}
\usepackage{url}

\catcode`\@=11
\@namedef{subjclassname@2010}{%
\textup{2010} Mathematics Subject Classification} \catcode`\@=12

\newtheorem{thm}{Theorem}[section]
\newtheorem{prop}{Proposition}[section]

\newtheorem{corr}{Corollary}[section]
\newtheorem{ex}{Example}[section]
\newtheorem{rem}{Remark}[section]

\numberwithin{equation}{section}

\allowdisplaybreaks

\newcommand{\N}{{\mathbb N}}
\newcommand{\C}{{\mathbb C}}
\newcommand{\BB}{{\mathbb B}}
\def\S{{\mathbb S}}             

\newcommand{\Ric}{\operatorname{Ric}}
\newcommand{\scal}{\operatorname{scal}}
\newcommand{\tr}{\operatorname{tr}}
\newcommand{\vol}{\operatorname{dvol}}

\def\T{{\mathcal T}}            
\def\V{{\mathcal V}}            

\def\Rho{{\sf P}}               

\def\J{{\sf J}}                 

\def\f{{\frac{n}{2}}}

\title{Holographic formula for $Q$-curvature. II}

\author{Andreas Juhl}

\address{Humboldt-Universit\"at, Institut f\"ur Mathematik, Unter den Linden,
D-10099 Berlin} \email{ajuhl@math.hu-berlin.de}

\begin{document}

\begin{abstract} We extend the holographic formula for
the critical $Q$-curvature in \cite{gj} to all $Q$-curvatures.
Moreover, we confirm a conjecture of \cite{juhl-book}.
\end{abstract}


\maketitle

\centerline \today

\tableofcontents

\footnotetext{The work was supported by SFB 647
``Space-Time-Matter'' of DFG.}

\section{Introduction and formulation of the main result}

The present work is a sequel to \cite{gj}, which gave a holographic
formula for Branson's critical $Q$-curvature.

The notion of $Q$-curvature was introduced by Branson in
\cite{bran}. Any Riemannian manifold $(M,g)$ of even dimension $n$
comes with a finite sequence $Q_2, Q_4, \dots, Q_n$ of
$Q$-curvatures. The quantity $Q_{2N} \in C^\infty(M)$ arises by
\begin{equation}\label{fund}
P_{2N}(g) (1) = (-1)^N \left(\f-N\right) Q_{2N}(g)
\end{equation}
through the constant term of the conformally covariant power $P_{2N}$
of the Laplacian constructed in \cite{GJMS}. $Q_{2N}$ is a curvature
invariant of order $2N$. For general metrics, the sequence $P_2, P_4,
\dots, P_n$ of GJMS-operators and the associated sequence of
$Q$-curvatures terminate at the {\em critical} GJMS-operator $P_n$ and
the {\em critical} $Q$-curvature $Q_n$, respectively. A subtle point
is that, by $P_n(g)(1)=0$, the critical $Q$-curvature $Q_n$ is not
defined by \eqref{fund}. Instead, it arises through a limiting
procedure from the subcritical $Q$-curvatures. The main feature which
distinguishes the critical $Q$-curvature from the subcritical ones is
its remarkable linear transformation property
$$
e^{n\varphi} Q_n(e^{2\varphi} g) = Q_n(g) + (-1)^\f P_n(g)(\varphi),
\; \varphi \in C^\infty(M)
$$
under conformal changes of the metric. Finally, in odd dimensions,
the sequences $P_2, P_4, \dots$ and $Q_2, Q_4, \dots$ both are
infinite. However, in that case, there is no critical $Q$-curvature.

The lowest order $Q$-curvatures are given by the explicit formulas
\begin{equation}
Q_2 = \frac{\scal}{2(n\!-\!1)} \quad \mbox{and} \quad Q_4 = \f \J^2
- 2 |\Rho|^2 - \Delta \J,
\end{equation}
where we use the notation
$$
\J = \frac{\scal}{2(n\!-\!1)} \quad \mbox{and} \quad \Rho =
\frac{1}{n\!-\!2} \left(\Ric - \J g\right)
$$
with $\Delta$ denoting the non-positive Laplacian. $\Rho$ is the
Schouten tensor. Among all $Q$-curvatures, these are by far the
mostly studied ones. For more information see \cite{BJ} and
\cite{juhl-book}.

The following theorem extends the main result of \cite{gj} to all
$Q$-curvatures.

\begin{thm}\label{main-form} Assume that $1 \le N \le \f$ for
even $n$ and $N \ge 1$ for odd $n$. Then
\begin{equation}\label{con}
4N c_N Q_{2N} = 2N v_{2N} + \sum_{j=1}^{N-1} (2N\!-\!2j)
\T_{2j}^* \left(\f\!-\!N\right) (v_{2N-2j})
\end{equation}
with $1/c_N = (-1)^N 2^{2N} N! (N\!-\!1)!$.
\end{thm}

The identities \eqref{con} express all $Q$-curvatures $Q_{2N}$ on
$M$ in terms of quantities which are associated to corresponding
Poincar\'e-Einstein metrics $g_+$ on $M \times (0,\varepsilon)$.
These are
\begin{itemize}
\item the holographic coefficients $v_{2j}$ and
\item the coefficients $\T_{2j}(n/2\!-\!N)$ of asymptotic expansions
of eigenfunctions of the Laplacian $\Delta_{g_+}$ for the eigenvalue
$N^2 \!-\!n^2/4$.
\end{itemize}
The relevant concepts will be recalled in Section \ref{basics}. The
role of geometry in one higher dimension in Theorem \ref{main-form}
motivates to refer to the formulas \eqref{con} as holographic.
Theorem \ref{main-form} (for even $n$) was formulated as Conjecture
6.9.1 in \cite{juhl-book}.

The paper is organized as follows. In Section \ref{basics}, we
recall basic concepts and establish some identities which will be
referred to as the master relations. These seem to be of independent
interest since they hold true in a wider setting. In Section
\ref{holo}, we show that the master relations for
Poincar\'e-Einstein metrics imply Theorem \ref{main-form}. For this
it suffices to combine the master relations for Poincar\'e-Einstein
metrics with the identities \eqref{property-1} and
\eqref{property-2}. This method also gives an alternative proof of
\eqref{con} in the critical case. In Section \ref{sphere}, we give
an independent proof of the master relation for the Poincar\'e
metric on the unit ball by using hypergeometric functions. As a
byproduct, it yields closed formulas for the $Q$-curvature
polynomials of round spheres.

\section{Master relations}\label{basics}

Although we shall work in the same framework as in \cite{gj}, we are
going to use the notation of \cite{juhl-book}. For the convenience
of the reader, we recall the basic definitions. For a given metric
$h$ on the manifold $M$ of even dimension $n$, let
\begin{equation}\label{PE1}
g_+ = r^{-2} (dr^2 + h_r)
\end{equation}
with
\begin{equation}\label{PE2}
h_r = h + r^2 h_{(2)} + \dots + r^{n-2} h_{(n-2)} + r^n (h_{(n)}+
\log r \bar{h}_{(n)}) + \cdots
\end{equation}
be a metric on $M \times (0,\varepsilon)$ so that the tensor
$\Ric(g_+) + n g_+$ satisfies the asymptotic Einstein condition
\begin{equation}\label{einstein}
\Ric(g_+) + n g_+ = O(r^{n-2})
\end{equation}
together with a certain vanishing trace condition. These conditions
uniquely determine the coefficients $h_{(2)}, \dots, h_{(n-2)}$.
They are given as polynomial formulas in terms of $h$, its inverse,
the curvature tensor of $h$, and its covariant derivatives. The
coefficient $\bar{h}_{(n)}$ and the quantity $\tr(h_{(n)})$ are
determined as well (here and in the following, traces are meant with
respect to $h$). Moreover, $\bar{h}_{(n)}$ is trace-free, and the
trace-free part of $h_{(n)}$ is undetermined. A metric $g_+$ with
these properties is called a Poincar\'e-Einstein metric with
conformal infinity $[h]$. For odd $n$, the condition
\eqref{einstein} can be satisfied to all orders by a formal series
$h_r = h + r^2 h_{(2)} + r^4 h_{(4)} + \cdots$ with {\em even}
powers, the coefficients of which are uniquely determined by $h$.
For full details see \cite{FG-final}.

The volume form of $g_+$ has the form
$$
\vol(g_+) = r^{-n-1} v(r) dr \vol(h),
$$
where $v(r) = \vol(h_r)/\vol(h)$. The coefficients in the Taylor
series
$$
v(r) = 1 + v_2 r^2 + v_4 r^4 + \cdots
$$
are known as {\em renormalized volume} coefficients (\cite{G-vol},
\cite{G-ext}) or {\em holographic} coefficients (\cite{juhl-book},
\cite{BJ}). We set $v_0=1$. For even $n$, the coefficients $v_{2j}
\in C^\infty(M)$, $j=1,\dots,\f$ are uniquely determined by $h$ and
are given by respective local formulas which involves at most $2j$
derivatives of the metric $h$. For odd $n$, all elements in the
corresponding infinite sequence $v_0, v_2, v_4, \dots$ are uniquely
determined by $h$.

Now we consider asymptotic expansions of eigenfunctions of the
Laplace-Beltrami operator of $g_+$, i.e., solutions of the equation
\begin{equation}\label{eigen}
-\Delta_{g_+} u = \lambda(n-\lambda)u, \; \lambda \in \C.
\end{equation}
For given $f \in C^\infty(M)$, we consider formal solutions of
\eqref{eigen} in form of power series
\begin{equation}\label{expansion}
r^\lambda a_0 + \sum_{N \ge 1} r^{\lambda+N} a_N \quad \mbox{with}
\quad  a_0 = f.
\end{equation}
It turns out that, for even $n$, the coefficients $a_N(h;\lambda)$
with odd $N \le n-1$ vanish, and that the coefficients
$a_2(h;\lambda), \dots, a_{n-2}(h;\lambda)$ with even indices are
uniquely determined by $a_0 = f$. They are given by respective
differential operators $\T_{2N}(h;\lambda)$ acting on $f$. The
operators $\T_{2N}(h;\lambda)$ are natural in the metric $h$ and
rational in $\lambda$. More precisely, any $\T_{2N}(h;\lambda)$ can be
written in the form
\begin{equation}\label{TP}
\T_{2N}(h;\lambda) = \frac{1}{2^{2N} N! (\f\!-\!\lambda\!-\!1)
\cdots (\f\!-\!\lambda\!-\!N)} P_{2N}(h;\lambda)
\end{equation}
with a polynomial (in $\lambda$) family $P_{2N}(h;\lambda) =
\Delta_h^N + LOT$; we recall that $\Delta$ denotes the non-positive
Laplacian. Note that \eqref{TP} shows that the poles of
$\T_{2N}(\lambda)$ are contained in the set
$$
\left\{\f-1,\dots,\f-N\right\}.
$$
Similarly, for odd $n$, the infinite sequence $\T_2(h;\lambda),
\T_4(h;\lambda), \dots$ is uniquely determined. Now for special
parameters $\lambda$, the families $P_{2N}(h;\lambda)$ contain the
GJMS-operators of $h$. In fact, for even $n$ and $1 \le N \le \f$,
we have the important relations
\begin{equation}\label{property-1}
P_{2N}\left(h;\f\!-\!N\right) = P_{2N}(h).
\end{equation}
The same relations hold true for odd $n$ and all $N \ge 1$. For details
see \cite{GZ}.

\begin{thm}\label{master} Assume that $1 \le N \le \f$ for even $n$
and $N \ge 1$ for odd $n$. Then the identities
\begin{equation}\label{master-3}
\lambda N \sum_{j=0}^N \T_{2j}^*(\lambda)(v_{2N-2j}) +
(\lambda\!-\!n\!+\!2N) \sum_{j=0}^N j \T_{2j}^*(\lambda)(v_{2N-2j}) = 0
\end{equation}
hold true as identities of rational functions.
\end{thm}

The identities \eqref{master-3} will be called the {\em master
relations}. A calculation shows that the master relations are
equivalent to the identities
\begin{equation}\label{master-2}
(\lambda\!-\!n\!+\!2N) \sum_{j=0}^N (2N\!+\!2j)
\T_{2j}^*(\lambda)(v_{2N-2j}) = -2N(n\!-\!2N) \sum_{j=0}^N
\T_{2j}^*(\lambda)(v_{2N-2j})
\end{equation}
of rational functions. In terms of the polynomials
\begin{equation}\label{Q-pol}
Q_{2N}^{res}(\lambda) = -2^{2N} N! \left(
\lambda\!+\!\f\!-\!2N\!+\!1 \right) \cdots \left(
\lambda\!+\!\f\!-\!N \right) \sum_{j=0}^N
\T_{2j}^*(\lambda\!+\!n\!-\!2N)(v_{2N-2j})
\end{equation}
and\footnote{In \cite{juhl-book} and \cite{juhl-power}, the
definition of $\V_{2N}(\lambda)$ involves a shift of the argument.}
\begin{equation}\label{V-pol-def}
\V_{2N}(\lambda) = \left(\lambda\!+\!\f\!-\!2N\!+\!1\right) \cdots
\left(\lambda\!+\!\f\!-\!N\right) \sum_{j=0}^N (2N\!+\!2j)
\T^*_{2j}(\lambda\!+\!n\!-\!2N)(v_{2N-2j})
\end{equation}
the relations \eqref{master-2}, in turn, can be formulated as the
identities
\begin{equation}\label{master-1}
2^{2N-2} (N\!-\!1)! \lambda \V_{2N}(\lambda) = \left(\f\!-\!N\right)
Q_{2N}^{res}(\lambda)
\end{equation}
of polynomials. These alternative versions will be referred to as
master relations, too. The polynomials $Q_{2N}^{res}(\lambda)$ were
introduced in \cite{juhl-book}. They are called the $Q$-curvature
polynomials. The identity \eqref{master-1} first appeared in Section
6 of \cite{Q8}.

We continue with the

\begin{proof}[Proof of Theorem \ref{master}] First assume that
$n$ is even. We choose $N$ as stated. For $f \in C_0^\infty(M)$ and
$\lambda \not\in \left\{\f-1,\dots,0\right\}$ we set
$$
u = r^\lambda f + \T_2(\lambda)(f) r^{\lambda+2} + \dots +
\T_n(\lambda)(f) r^{\lambda+n}.
$$
Then
\begin{equation}\label{eigen-app}
\Delta_{g_+} u  + \lambda(n\!-\!\lambda) u \in O(r^{\lambda+n+1}).
\end{equation}
We choose small numbers $\varepsilon$ and $\delta$ so that $0 <
\varepsilon < \delta$ and consider the asymptotic expansion of the
integral
\begin{equation}\label{integral}
\int_{\varepsilon < r < \delta} \Delta_{g_+} u \vol(g_+)
\end{equation}
for $\varepsilon \to 0$. Assuming that $\lambda + 2N-n \ne 0$, we
determine the coefficient of $\varepsilon^{\lambda+2N-n}$ in this
expansion in two different ways. On the one hand, we evaluate the
integral
$$
\int_{\varepsilon < r < \delta} u \vol(g_+) =
\int_\varepsilon^\delta \int_M u v(r) r^{-n-1} dr \vol(h)
$$
by plugging in the asymptotic expansions of $u$ and $v(r)$. This
yields the contribution
$$
-\frac{\varepsilon^{\lambda+2N-n}}{\lambda\!+\!2N\!-\!n} \left(
\sum_{j=0}^N \int_M \T_{2j}(\lambda)(f) v_{2N-2j} \vol(h) \right).
$$
On the other hand, Green's formula shows that
$$
-\int_{\varepsilon < r < \delta} \Delta_{g_+} u \vol(g_+) = \left(
\int_{r = \varepsilon} + \int_{r = \delta} \right) \frac{\partial
u}{\partial \nu} r^{-n} v(r) \vol(h),
$$
where $\nu$ denotes the inward unit normal (with respect to $g_+$).
The coefficient of $\varepsilon^{\lambda+2N-n}$ in the asymptotic of
this integral is given by
$$
\sum_{j=0}^N (\lambda\!+\!2j) \int_M \T_{2j}(\lambda)(f) v_{2N-2j}
\vol(h).
$$
Now \eqref{eigen-app} implies that
$$
-\frac{\lambda(n\!-\!\lambda)}{\lambda\!+\!2N\!-\!n} \sum_{j=0}^N
\int_M \T_{2j}(\lambda)(f) v_{2N-2j} \vol(h) = \sum_{j=0}^N
(\lambda\!+\!2j) \int_M \T_{2j}(\lambda)(f) v_{2N-2j} \vol(h)
$$
for all $f \in C_0^\infty(M)$. We apply partial integration and find
the pointwise identity
$$
-\frac{\lambda(n\!-\!\lambda)}{\lambda\!+\!2N\!-\!n}
\sum_{j=0}^N \T_{2j}^*(\lambda)(v_{2N-2j}) =
\sum_{j=0}^N (\lambda\!+\!2j) \T_{2j}^*(\lambda)(v_{2N-2j}).
$$
The latter equation is equivalent to \eqref{master-1}. Finally, let
$n$ be odd and $N \ge 1$. For $f \in C_0^\infty(M)$ and $\lambda
\not\in \f- \N$ we set
$$
u = r^\lambda f + \T_2(\lambda)(f) r^{\lambda+2} + \dots +
\T_{2N}(\lambda)(f) r^{\lambda+2N}.
$$
The assertion follows by analogous consideration of the coefficient
of $\varepsilon^{\lambda+2N-n}$ in the asymptotic expansion of
\eqref{integral}.
\end{proof}

As an illustration, we make explicit the master relations for $N=1$
and $N=2$ for Poincar\'e-Einstein metrics.

\begin{ex}\label{m-2} For $N=1$, \eqref{master-3} reads
\begin{equation}\label{m-1}
\lambda \left( v_2 + \T_2^*(\lambda)(1) \right) +
(\lambda\!-\!n\!+\!2) \T_2^*(\lambda)(1) = 0.
\end{equation}
It is is straightforward to verify \eqref{m-1} by using
\begin{equation}\label{v2}
v_2 = - \frac{1}{2} \J
\end{equation}
and
\begin{equation}\label{T2}
\T_2(\lambda) = \frac{1}{2(n\!-\!2\!-\!2\lambda)} (\Delta - \lambda
\J).
\end{equation}
\end{ex}

\begin{ex}\label{m-4} For $N=2$, we have the explicit formulas
\begin{equation}\label{v4}
v_4 = \frac{1}{8} (\J^2 - |\Rho|^2)
\end{equation}
and
\begin{multline}\label{T4}
\T_4(\lambda) =
\frac{1}{8(n\!-\!2\!-\!2\lambda)(n\!-\!4\!-\!2\lambda)} \big[
(\Delta - (\lambda\!+\!2)\J)(\Delta - \lambda \J) \\ +
\lambda(2\lambda\!-\!n\!+\!2) |\Rho|^2 + 2(2\lambda\!-\!n\!+\!2)
\delta (\Rho d) + (2\lambda\!-\!n\!+\!2) (d\J,d) \big]
\end{multline}
(see \cite{juhl-book} for details). Direct calculations show that
\begin{multline*}
8 \T_4^*(\lambda)(1) + 6 \T_2^*(\lambda)(v_2) + 4 v_4 \\
= \left(\f\!-\!2\right)
\frac{1}{(n\!-\!2\!-\!2\lambda)(n\!-\!4\!-\!2\lambda)} \left[
\lambda (2|\Rho|^2 \!-\! \J^2) + (n\!-\!2) (\J^2 \!-\! |\Rho|^2) -
\Delta \J \right]
\end{multline*}
and
\begin{multline*}
\T_4^*(\lambda)(1) + \T_2^*(\lambda)(v_2) + v_4 \\
= - \frac{1}{8} (\lambda\!-\!n\!+\!4)
\frac{1}{(n\!-\!2\!-\!2\lambda)(n\!-\!4\!-\!2\lambda)} \left[
\lambda (2|\Rho|^2 \!-\! \J^2) + (n\!-\!2) (\J^2 \!-\! |\Rho|^2) -
\Delta \J \right].
\end{multline*}
These results confirm \eqref{master-2}.
\end{ex}

\begin{rem}\label{AH} The proof of Theorem \ref{master} does not utilize
the property that $g_+$ is Einstein. In fact, the master relations
remain true for metrics $g_+$ of the form
$$
g_+ = r^{-2}(dr^2 + h_r)
$$
with even one-parameter families $h_r$. In particular, for $h_r = h
+ r^2 h_{(2)} + r^4 h_{(4)} + \dots$, the results of Example
\ref{m-2} and Example \ref{m-4} extend if $v_2$ and $v_4$ are
replaced by
$$
v_2 = - \frac{1}{2} \tr (h_{(2)}) \quad \mbox{and} \quad v_4 =
\frac{1}{2} \tr (h_{(4)}) - \frac{1}{4} \tr (h_{(2)}^2) +
\frac{1}{8} \tr (h_{(2)})^2,
$$
and $\J$ and $|\Rho|^2$ are replaced by $\tr (h_{(2)})$ and $-4 \tr
(h_{(4)}) + 2 \tr (h_{(2)}^2)$.
\end{rem}

Some special cases of the master relations are of particular
interest. First of all, for $\lambda=0$, the master relation
\eqref{master-1} states that
\begin{equation}\label{Q-van}
Q_{2N}^{res}(0) = 0.
\end{equation}
An independent proof of this result can be found in \cite{BJ}
(Theorem 1.6.6).

The case $\lambda = \f-N$ will be considered in Section \ref{holo}.
It leads to holographic formulas for $Q$-curvature.

In the critical case $2N=n$, Theorem \ref{master} implies the
following vanishing result. We recall that
\begin{equation}
\V_n(\lambda) = \left[\left(\lambda\!-\!\f\!+\!1\right) \cdots
\lambda\right] \sum_{j=0}^\f (n\!+\!2j)
\T^*_{2j}(\lambda)(v_{n-2j})
\end{equation}
by \eqref{V-pol-def}.

\begin{thm}\label{V-van} $\V_n(\lambda) = 0$.
\end{thm}

Theorem \ref{V-van} confirms Conjecture 6.11.2 in \cite{juhl-book}
and shows that the assumption in Theorem 6.11.15 is vacuous.

By the definitions, both polynomials $Q_{2N}^{res}(\lambda)$ and
$\V_{2N}(\lambda)$ have degree $\le N$. However, the master relation
\eqref{master-1} implies that $\V_{2N}$, in fact, has only degree $\le
N-1$. The latter fact was proved in \cite{Q8}, where it plays a
central role in the proof of a universal recursive formula for the
$Q$-curvature $Q_8$.

Since \eqref{master-3} takes the form
$$
2 N \left(2\lambda\!-\!n\!+\!2N\right) \T_{2N}^*(\lambda)(v_0) +
\dots + 2\lambda N v_{2N} = 0,
$$
it implies a formula for $P_{2N}^*(\lambda)(1)$ as a linear
combination (with coefficients depending on $\lambda$) of
$$
P_{2j}^*(\lambda)(v_{2N-2j}) \quad \mbox{for $j=0,\dots,N-1$}.
$$
In the critical case, this observation was noticed and used in
\cite{juhl-book} for low orders.

\section{Holographic formulas for $Q$-curvatures}\label{holo}

In the present section, we prove Theorem \ref{main-form}. We restate
the result in the following form.

\begin{thm}\label{main} Assume that $1 \le N \le \f$ for even $n$
and $N \ge 1$ for odd $n$. Then
\begin{equation}\label{Q-sub-holo}
4N c_N Q_{2N} = \sum_{j=0}^{N-1} (2N\!-\!2j)
\T_{2j}^*\left(\f\!-\!N\right)(v_{2N-2j})
\end{equation}
with
$$
c_N = (-1)^{N} \frac{1}{2^{2N}N!(N\!-\!1)!}.
$$
\end{thm}

\begin{proof} We write \eqref{master-3} in the form
\begin{equation}\label{master-split}
\sum_{j=0}^{N-1} (2\lambda N + 2(\lambda\!-\!n\!+\!2N)j)
\T_{2j}^*(\lambda)(v_{2N-2j}) + 2N (2\lambda\!-\!n\!+\!2N)
\T_{2N}^*(\lambda)(1) = 0.
\end{equation}
The families $\T_{2j}(\lambda)$ for $j=0,\dots,N-1$ are regular at
$\lambda=\f-N$. Moreover, the relation
$$
\left(\lambda\!-\!\f\!+\!N\right) \T_{2N}^*(\lambda) = -
\frac{1}{2^{2N}N!}
\frac{1}{(\f\!-\!\lambda\!-\!1)\cdots(\f\!-\!\lambda\!-\!N\!+\!1)}
P_{2N}^*(\lambda)
$$
shows that the product $\left(\lambda\!-\!\f\!+\!N\right)
\T_{2N}^*(\lambda)$ is regular at $\lambda=\f-N$. By
\eqref{property-1} and the self-adjointness of $P_{2N}$, we have
$$
P_{2N}^*\left(\f\!-\!N\right)(1) = P_{2N}^*(1) = P_{2N}(1) = (-1)^N
\left(\f\!-\!N\right) Q_{2N}.
$$
It follows that the value of $\left(\lambda\!-\!\f\!+\!N\right)
\T_{2N}^*(\lambda)(1)$ at $\lambda=\f-N$ equals
$$
-\frac{1}{2^{2N}N!(N\!-\!1)!} P_{2N}^*\left(\f\!-\!N\right)(1) =
(-1)^{N-1} \left(\f\!-\!N\right) \frac{1}{2^{2N}N!(N\!-\!1)!}
Q_{2N}.
$$
Hence for $\lambda=\f-N$, the master relation \eqref{master-split}
states that
\begin{equation*}
\left(\f\!-\!N\right) \sum_{j=0}^{N-1} (2N\!-\!2j)
\T_{2j}^*\left(\f\!-\!N\right)(v_{2N-2j}) = (-1)^N
\left(\f\!-\!N\right) \frac{4N}{2^{2N}N!(N\!-\!1)!} Q_{2N}.
\end{equation*}
Now assume that $n$ is even and $2N < n$. We divide the latter
relation by $\f-N$ and find
\begin{equation}\label{holo-sub}
\sum_{j=0}^{N-1} (2N\!-\!2j)
\T_{2j}^*\left(\f\!-\!N\right)(v_{2N-2j}) = 4N c_N Q_{2N}
\end{equation}
with
$$
c_N = (-1)^{N} \frac{1}{2^{2N}N!(N\!-\!1)!}.
$$
This proves the assertion in the subcritical case. The same argument
completes the proof in odd dimensions. A formal continuation of
\eqref{holo-sub} to $2N=n$ yields the holographic formula
\begin{equation}\label{holo-crit}
\sum_{j=0}^{\f-1} (n\!-\!2j) \T_{2j}^*(0)(v_{n-2j}) = 2n c_\f Q_n
\end{equation}
for the critical $Q$-curvature $Q_n$ \cite{gj}, \cite{juhl-book}. The
above proof, however does {\em not} work in this case since it
involves a division by $\f-N$. In fact, in the critical case, the
master relation \eqref{master-3} states that
\begin{equation}\label{master-crit}
\lambda \sum_{j=0}^\f (n\!+\!2j) \T_{2j}^*(\lambda)(v_{n-2j}) = 0.
\end{equation}
We have seen above that the left-hand side of the critical master
relation \eqref{master-crit} is regular at $\lambda=0$ and vanishes at
$\lambda = 0$ by trivial reasons. In order to derive
\eqref{holo-crit}, we differentiate \eqref{master-crit} at
$\lambda=0$. Separating the last term, we find
$$
\sum_{j=0}^{\f-1} (n\!+\!2j) \T_{2j}^*(0)(v_{n-2j}) + 2n
(d/d\lambda)|_0(\lambda \T_n^*(\lambda)(1)) = 0.
$$
Now
$$
\lambda \T_n(\lambda) = - \frac{1}{2^n \left(\f\right)!}
\frac{1}{(\f\!-\!\lambda\!-\!1)\cdots(-\lambda\!+\!1)} P_n(\lambda).
$$
Hence
$$
\sum_{j=0}^{\f-1} (n\!+\!2j) \T_{2j}^*(0)(v_{n-2j}) - \frac{2n}{2^n
\left(\f\right)!(\f\!-\!1)!} \dot{P}_n^*(0)(1) = 0.
$$
We combine this result with the identity
$$
n \left(\dot{P}_n^*(0) - \dot{P}_n(0) \right)(1) = 2^n
\left(\f\right)! \left(\f\!-\!1\right)! \sum_{j=0}^{\f-1} 2j
\T_{2j}^*(0)(v_{n-2j})
$$
(see \cite{gj}, Proposition 2 or \cite{juhl-book}, Theorem 6.6.4).
It follows that
$$
\sum_{j=0}^{\f-1} (n\!-\!2j) \T_{2j}^*(0)(v_{n-2j}) = \frac{2n}{2^n
\left(\f\right)!\left(\f\!-\!1\right)!} \dot{P}_n(0)(1).
$$
Now the relation (see \cite{GZ})
\begin{equation}\label{property-2}
\dot{P}_n^*(0) = (-1)^\f Q_n
\end{equation}
implies the holographic formula \eqref{holo-crit}.
\end{proof}

We add some comments.

Eq.~\eqref{Q-sub-holo} expresses the difference
$$
Q_{2N} - (-1)^N 2^{2N-1} N! (N\!-\!1)! v_{2N}
$$
in terms of the lower order constructions $\T_{2j}(n/2-N)$ and
$v_{2j}$ for $j=0,\dots,N\!-\!1$. The same difference can be
expressed also in terms of lower order GJMS-operators and lower
order $Q$-curvatures. For the details we refer to \cite{juhl-power}
and \cite{Q8}.

The vanishing result $\V_n(\lambda)=0$ (see Theorem \ref{V-van})
implies
$$
n \sum_{j=0}^\f \dot{\T}_{2j}^*(0)(v_{n-2j}) + \sum_{j=0}^\f 2j
\dot{\T}_{2j}^*(0)(v_{n-2j}) = 0.
$$
In combination with the identity $\dot{Q}_n^{res}(0) = -(-1)^\f Q_n$
(see \cite{gj} and \cite{juhl-book}) we find
$$
c_\f^{-1} \frac{1}{n} \sum_{j=0}^\f 2j \dot{\T}_{2j}^*(0)(v_{n-2j})
= - \frac{1}{2} \ddot{Q}_n^{res}(0) - \left(\sum_{j=1}^{\f-1}
\frac{1}{k} \right) Q_n.
$$
For an application of the latter relation we refer to Theorem
6.11.15 of \cite{juhl-book} (see also the remarks around (2.6) in
\cite{gj}).

We close this section with a brief discussion of two examples of
Theorem \ref{main}.

\begin{ex} In dimension $n \ge 4$, the holographic formula for
$Q_4$ states that
\begin{equation}\label{holo-Q4}
\frac{1}{4} Q_4 = 4 v_4 + 2 \T_2^* \left(\f\!-\!2\right) (v_2).
\end{equation}
An easy calculation using \eqref{v2}--\eqref{v4} shows that
\eqref{holo-Q4} is equivalent to the familiar expression
$$
Q_4 = \f \J^2 - 2|\Rho|^2 - \Delta \J.
$$
\end{ex}

\begin{ex} In dimension $n \ge 6$, the holographic formula for $Q_6$
states that
\begin{equation}\label{holo-Q6}
-\frac{1}{2^6}  Q_6 = 6 v_6 + 4 \T_2^*\left(\f\!-\!3\right) (v_4) +
2 \T_4^*\left(\f\!-\!3\right) (v_2).
\end{equation}
Explicit formulas for $\T_2(\lambda)$ and $\T_4(\lambda)$ are
displayed in Example \ref{m-4}. For a detailed comparison of
\eqref{holo-Q6} with alternative explicit formulas for $Q_6$ we
refer to \cite{juhl-book} (Theorem 6.10.4).
\end{ex}

\section{Master relations and $Q$-curvature polynomials for round
  spheres}\label{sphere}

In the present section, we prove the master relations for round
spheres $\S^n$ by using hypergeometric identities. More precisely,
we establish \eqref{master-3} by comparing explicit formulas for
both sides of \eqref{master-2}. These results confirm closed
formulas for the $Q$-curvature polynomials found in
\cite{juhl-power} by different methods.

The arguments rest on the following result.

\begin{prop}\label{prop-1} On the round sphere $\S^n$,
\begin{equation}
P_{2N}(\lambda)(1) = P_{2N}^*(\lambda)(1) = (-1)^N \left(\f\right)_N
(\lambda)_N
\end{equation}
for all $N \ge 0$. Here $(x)_N = x(x\!+\!1)\cdots(x\!+\!N\!-\!1)$ is
the usual Pochhammer symbol.
\end{prop}

\begin{proof} On the round sphere $\S^n$, the operators $P_{2N}(\lambda)$
are polynomials in the Laplacian. In particular, they are
self-adjoint. Now, by the definitions, the assertion is equivalent
to
\begin{equation}\label{claim}
\T_{2N}(\lambda)(1) = \frac{(\f)_N}{2^{2N} N!}
\frac{(\lambda)_N}{(\lambda\!-\!\f\!+\!1)_N}
\end{equation}
for all $N \ge 0$. In fact, it will be more convenient to prove the
equivalent assertion that
\begin{equation}\label{claim-2}
\T_{2N}\left(\frac{1}{4}g_c,\lambda\right)(1) = \frac{(\f)_N}{N!}
\frac{(\lambda)_N}{(\lambda\!-\!\f\!+\!1)_N}
\end{equation}
for all $N \ge 0$. Note that \eqref{claim-2} is equivalent to
\begin{equation}\label{claim-3}
\sum_{N \ge 0} \T_{2N} \left(\frac{1}{4} g_c;\lambda \right)(1)
s^{\lambda + 2N} = s^\lambda {_2F_1}
\left(\f,\lambda;\lambda\!-\!\f\!+\!1;s^2\right),
\end{equation}
where
$$
{_2F_1}(a,b;c;x) = \sum_{n \ge 0} \frac{(a)_n(b)_n}{(c)_n}
\frac{x^n}{n!}
$$
is the Gau{\ss} hypergeometric function. In the following, we shall
derive \eqref{claim-3} from the well-known fact that the radial
eigenfunctions of the Laplacian of the hyperbolic metric
$$
\frac{4}{(1-|x|^2)^2} \sum_{i=1}^n dx_i^2
$$
on the unit ball $\BB^n$ with boundary $S^{n-1}$ are constant
multiples of
$$
u(r) = (1\!-\!r^2)^\lambda {_2F_1}
\left(\lambda,\lambda\!-\!\f\!+\!1;\f;r^2\right)
$$
(see \cite{BJ}, Section 1.4). The substitution
$$
s = \frac{1-r}{1+r}, \; r = |x|
$$
brings the hyperbolic metric into the form
$$
s^{-2} \left(ds^2 + \frac{1}{4} (1\!-\!s^2)^2 g_c \right).
$$
Let
$$
v(s) = u\left(\frac{1-s}{1+s}\right).
$$
We use the identity (see eq.~(1) in Section 2.10 of \cite{BE})
\begin{multline*}
{_2F_1}(a,b;c;x) = \alpha {_2F_1}(a,b;a+b-c+1;1-x) \\ + \beta
(1-x)^{c-a-b} {_2F_1}(c-a,c-b;c-a-b+1;1-x)
\end{multline*}
with
$$
\alpha = \frac{\Gamma(c)\Gamma(c-a-b)}{\Gamma(c-a)\Gamma(c-b)} \quad
\mbox{and} \quad \beta =
\frac{\Gamma(c)\Gamma(a+b-c)}{\Gamma(a)\Gamma(b)}
$$
to write $u$ in the form
\begin{multline*}
u(r) = A (1\!-\!r^2)^\lambda {_2F_1}
\left(\lambda,\lambda\!-\!\f\!+\!1;2\lambda\!-\!n\!+\!2;1\!-\!r^2\right) \\
+ B (1\!-\!r^2)^{n-1-\lambda} {_2F_1}
\left(\f\!-\!\lambda,n\!-\!\lambda\!-\!1;n\!-\!2\lambda;1\!-\!r^2\right)
\end{multline*}
with
$$
A = \frac{\Gamma(\f)
\Gamma(n\!-\!2\lambda\!-\!1)}{\Gamma(\f\!-\!\lambda)
\Gamma(n\!-\!\lambda\!-\!1)} \quad \mbox{and} \quad B =
\frac{\Gamma(\f)\Gamma(2\lambda\!-\!n\!+\!1)}{\Gamma(\lambda)\Gamma(\lambda\!-\!\f\!+\!1)}.
$$
Now we substitute
$$
r = \frac{1-s}{1+s}.
$$
The relation (see (3.1.11) in \cite{AAR})
$$
{_2F_1} \left(a,b;2b;\frac{4x}{(1+x)^2}\right) = (1+x)^{2a} {_2F_1}
\left(a,a+\frac{1}{2}-b;b+\frac{1}{2};x^2\right)
$$
implies that
$$
{_2F_1} \left
(\lambda,\lambda\!-\!\f\!+\!1;2\lambda\!-\!n\!+\!2;\frac{4s}{(1\!+\!s)^2}\right)
= (1+s)^{2\lambda} {_2F_1} \left
(\lambda,\frac{n\!-\!1}{2};\lambda\!-\!\frac{n\!-\!3}{2};s^2\right)
$$
and \begin{multline*}
{_2F_1}\left(\f\!-\!\lambda,n\!-\!\lambda\!-\!1;n\!-\!2\lambda;\frac{4s}{(1\!+\!s)^2}\right)
\\ = (1+s)^{2n-2-2\lambda} {_2F_1}
\left(\frac{n\!-\!1}{2},n\!-\!\lambda\!-\!1;\frac{n\!+\!1}{2}\!-\!\lambda;s^2\right).
\end{multline*}
It follows that
\begin{multline*}
v(s) = A (4s)^\lambda {_2F_1}
\left(\lambda,\frac{n\!-\!1}{2};\lambda\!-\!\frac{n\!-\!3}{2};s^2\right)
\\ + B (4s)^{n-1-\lambda} {_2F_1}
\left(\frac{n\!-\!1}{2},n\!-\!\lambda\!-\!1;\frac{n\!+\!1}{2}\!-\!\lambda;s^2\right).
\end{multline*}
But the first term in this sum is a multiple of the right-hand side
of \eqref{claim-3} (after shifting $n$ by one). This completes the
proof. \end{proof}

Next, we calculate the sum on the right-hand side of
\eqref{master-2}.

\begin{prop}\label{prop-2} On the round sphere $\S^n$,
\begin{equation}\label{sum-1}
\sum_{j=0}^N \T_{2j}^*(\lambda)(v_{2N-2j}) = (-1)^N
\frac{(\f)_N}{2^{2N} N!} \frac{(\lambda\!-\!n\!+\!2N)
(\lambda\!-\!n\!+\!1)_{N-1}}{ (\lambda\!-\!\f\!+\!1)_N}.
\end{equation}
\end{prop}

\begin{proof} The formula
$$
g_+ = r^{-2} \left(dr^2 + \left(1-r^2/4 \right)^2 g_{S^n} \right)
$$
implies $v(r) = (1-r^2/4)^n$. Thus, the Taylor coefficients of
$v(r)$ are given by
\begin{equation}\label{hol-c}
v_{2N} = (-1)^N \frac{1}{2^{2N}} \binom{n}{N}.
\end{equation}
But since the holographic coefficients are constant, the left-hand
side of \eqref{sum-1} equals
\begin{equation}\label{sum-2}
\sum_{j=0}^N \T_{2j}^*(\lambda)(1) v_{2N-2j}.
\end{equation}
Thus, by Proposition \ref{prop-1} it suffices to verify that
\begin{equation*}
\sum_{j=0}^N \frac{(\f)_j (\lambda)_j}{2^{2j} j!
(\lambda\!-\!\f\!+\!1)_j} (-1)^{N-j} \frac{1}{2^{2N-2j}}
\binom{n}{N\!-\!j} = (-1)^N \frac{(\f)_N}{2^{2N} N!}
\frac{(\lambda\!-\!n\!+\!2N)(\lambda\!-\!n\!+\!1)}{(\lambda\!-\!\f\!+\!1)_N},
\end{equation*}
i.e.,
\begin{equation}\label{claim-red}
\sum_{j=0}^N \binom{n}{N\!-\!j} (-1)^j
\frac{(\f)_j(\lambda)_j}{(\lambda\!-\!\f\!+\!1)_j} \frac{1}{j!} =
\frac{(\f)_N}{N!} \frac{(\lambda\!-\!n\!+\!2N)
(\lambda\!-\!n\!+\!1)_{N-1}}{(\lambda\!-\!\f\!+\!1)_N}.
\end{equation}
For the proof of the summation formula \eqref{claim-red} we write
the left-hand side in hypergeometric notation:
\begin{multline*}
\binom{n}{N} \sum_{j=0}^N \frac{(\f)_N
(\lambda)_j}{(\lambda\!-\!\f\!+\!j)_j}
\frac{(-N)_j}{(n\!-\!N\!+\!1)_j} \frac{1}{j!} = \binom{n}{N} {_3F_2}
\left(\f,\lambda,-N;\lambda\!-\!\f\!+\!1,n\!-\!N\!+\!1;1\right).
\end{multline*}
In view of
$$
\f\!+\!\lambda\!-\!N\!+\!2 = \left(\lambda\!-\!\f\!+\!1\right) +
(n\!-\!N\!+\!1),
$$
this is a $2$-balanced hypergeometric sum. We evaluate this sum by
using Sheppard's formula (see Corollary 3.3.4 in \cite{AAR})
\begin{multline}\label{shep}
{_3F_2}(-n,a,b;d,e;1) = \frac{(d-a)_n(e-a)_n}{(d)_n(e)_n} \\
\times {_3F_2}(-n,a,a+b-n-d-e+1;a-n-d+1,a-n-e+1;1).
\end{multline}
We find
\begin{multline*}
\binom{n}{N} {_{3}F_2}
\left(-N,\f,\lambda;\lambda\!-\!\f\!+\!1,n\!-\!N\!+\!1;1\right)
\\ = \binom{n}{N} \frac{(\lambda\!-\!n\!+\!1)_N (\f\!-\!N\!+\!1)_N}{(\lambda\!-\!\f\!+\!1)_N
(n\!-\!N\!+\!1)_N} {_{3}F_2} \left(-N,\f,-1;n\!-\!N\!-\!\lambda,-\f;1\right) \\
= \frac{(\f)_N}{N!}
\frac{(\lambda\!-\!n\!+\!2N)(\lambda\!-\!n\!+\!1)_{N-1}}{(\lambda\!-\!\f\!+\!1)_N}.
\end{multline*}
by using
$$
{_{3}F_2} \left(-N,\f,-1;n\!-\!N\!-\!\lambda,-\f;1\right) = 1 +
\frac{N}{\lambda\!-\!n\!+\!N}  =
\frac{\lambda\!-\!n\!+\!2N}{\lambda\!-\!n\!+\!N}.
$$
The proof is complete. \end{proof}

Proposition \ref{prop-2} implies an explicit formula for the
$Q$-curvature polynomial \eqref{Q-pol}.

\begin{corr}\label{corr-1} On the round sphere $\S^n$,
$$
Q^{res}_{2N}(\lambda) = (-1)^{N-1} \prod_{j=0}^{N-1} \left(\f\!-\!j\right)
\lambda \prod_{j=1}^{N-1} (\lambda\!-\!N\!-\!j).
$$
\end{corr}

\begin{proof} Proposition \ref{prop-2} yields
$$
Q_{2N}^{res}(\lambda\!-\!n\!+\!2N) = (-1)^{N-1} \left(\f\right)_N
(\lambda\!-\!n\!+\!2N) (\lambda\!-\!n\!+\!1)_{N-1}.
$$
This formula implies the assertion.
\end{proof}

Corollary \ref{corr-1} was derived in \cite{juhl-power} (see Lemma 9.2)
by using the factorization identities for $Q$-curvature polynomials.

Next, we calculate the sum on the left-hand side of
\eqref{master-2}. By \eqref{claim}, we obtain
$$
\sum_{j=0}^N j \T_{2j}^*(\lambda)(v_{2N-2j}) = \frac{(-1)^N}{2^{2N}}
\binom{n}{N} \sum_{j=0}^N \frac{(\f)_N (\lambda)_j}
{(\lambda\!-\!\f\!+\!j)_j} \frac{(-N)_j}{(n\!-\!N\!+\!1)_j} j
\frac{1}{j!}.
$$
In terms of hypergeometric notation this sum equals
\begin{multline*}
\frac{(-1)^N}{2^{2N}} \binom{n}{N} \frac{\f \lambda (-N)}{(\lambda-\f+1)(n-N+1)} \\
\times
{_{3}F_2}\left(\f+1,\lambda+1,-N+1;\lambda\!-\!\f\!+\!2,n\!-\!N\!+\!2;1\right).
\end{multline*}
The ${_3F_2}$ is a $1$-balanced hypergeometric sum. By the formula
of Pfaff-Saalsch\"{u}tz (see Theorem 2.2.6 in \cite{AAR}) or
\eqref{shep}, we find
$$
{_{3}F_2}\left(\f+1,\lambda+1,-N+1;\lambda\!-\!\f\!+\!2,n\!-\!N\!+\!2;1\right)
= \frac{(\lambda\!-\!n\!+\!1)_{N-1}
(-\f\!+\!1)_{N-1}}{(\lambda\!-\!\f\!+\!2)_{N-1}(-n)_{N-1}}.
$$
Hence
\begin{align*}
\sum_{j=0}^N j \T_{2j}^*(\lambda)(v_{2N-2j}) & =
\frac{(-1)^{N-1}}{2^{2N}} \binom{n}{N} \frac{\f N}{n\!-\!N\!+\!1}
\frac{(-\f\!+\!1)_{N-1}}{(-n)_{N-1}}
\frac{\lambda(\lambda\!-\!n\!+\!1)_{N-1}}{(\lambda\!-\!\f\!+\!1)_N} \\
& = \frac{(-1)^{N-1}}{2^{2N}} \frac{(\f)_N}{(N\!-\!1)!} \lambda
\frac{(\lambda\!-\!n\!+\!1)_{N-1}}{(\lambda\!-\!\f\!+\!1)_N}.
\end{align*}
It follows that the second sum in \eqref{master-3} equals
$$
(-1)^{N-1} \frac{(\f)_N}{2^{2N}(N\!-\!1)!} \lambda
(\lambda\!-\!n\!-\!2N)
\frac{(\lambda\!-\!n\!+\!1)_{N-1}}{(\lambda\!-\!\f\!+\!1)_N}.
$$
On the other hand, by Proposition \ref{prop-2}, the first sum in
\eqref{master-3} equals
$$
(-1)^N \frac{(\f)_N}{2^{2N} (N\!-\!1)!} \lambda
(\lambda\!-\!n\!+\!2N) \frac{(\lambda\!-\!n\!+\!1)_{N-1}}{
(\lambda\!-\!\f\!+\!1)_N}.
$$
This completes the proof of \eqref{master-3} for $S^n$.

Finally, we note that the discussion in Section 7.9 of
\cite{FG-final} yields explicit formulas for the families
$P_{2N}(\lambda)$ on $\S^n$.\footnote{The author is grateful to R.
Graham for pointing this out.} These formulas can be used to give an
alternative proof of Proposition \ref{prop-1}.


\end{document}